\documentclass[a4paper]{article}
\usepackage{amsthm,amssymb,amsmath}
\usepackage[bb=boondox]{mathalfa}
\usepackage{bbm}
\newcommand{\1}{\mathbbm{1}}
\newcommand{\0}{\mathbb{0}}
\usepackage{todonotes}
\usepackage{graphicx}
\usetikzlibrary{calc}
\usepackage{todonotes}
\usepackage[T1]{fontenc}
\usepackage[utf8]{inputenc}

\usepackage{authblk}
\usepackage{soul}

\newtheorem{theorem}{Theorem}
\newtheorem{lemma}{Lemma}

\newtheorem{remark}{Remark}

\providecommand{\keywords}[1]
{
	\small	
	\textbf{\textit{Keywords---}} #1
}

\title{On the rank of the distance matrix of graphs}

\date{}

\author[1,2]{Ezequiel Dratman\thanks{edratman@campus.ungs.edu.ar}}
\author[1,2]{Luciano N. Grippo\thanks{lgrippo@campus.ungs.edu.ar}}
\author[1]{Ver{\'o}nica Moyano\thanks{vmoyano@campus.ungs.edu.ar}}
\author[1,3]{Adri{\'a}n Pastine\thanks{agpastine@unsl.edu.ar}}
\affil[1]{Consejo Nacional de Investigaciones Cient\'ificas y Tecnicas, Argentina.}
\affil[2]{Instituto de Ciencias\\Universidad Nacional de General Sarmiento, Argentina.}
\affil[3]{Instituto de Matem\'atica Aplicada San Luis, Universidad Nacional de San Luis, Argentina.}
\begin{document}

\maketitle

\begin{abstract}
		Let $G$ be a connected graph with $V(G)=\{v_1,\ldots,v_n\}$. The $(i,j)$-entry of the distance matrix $D(G)$ of $G$ is the distance between $v_i$ and $v_j$. In this article, using the well-known Ramsey's theorem, we prove that for each integer $k\ge 2$,  there is a finite amount of graphs whose distance matrices have rank $k$. We exhibit the list of graphs with distance matrices of rank $2$ and $3$. Besides, we study the rank of the distance matrices of graphs belonging to a family of graphs with their diameters at most two, the trivially perfect graphs.
		We show that for each $\eta\ge 1$ there exists a trivially perfect graph with nullity $\eta$. 
		We also show that for threshold graphs, which are a subfamily of the family of trivially perfect graphs, the nullity is bounded by one.       
\end{abstract}

\keywords{Distance Matrix, Distance Rank, Threshold Graph, Trivially Perfect Graph.}

\section{Introduction}

All graphs mentioned in this article are finite and have neither loops nor multiple edges. Let $G$ be a connected graph 
on $n$ vertices with vertex set $V =\{v_1,\dots,v_n\}$. The distance in $G$ between vertices $v_i$ and $v_j$ , denoted 
$d_G(v_i, v_j )$, is the number of edges of a shortest path linking $v_i$ and $v_j$. When the graph $G$ is clear
from the context we write $d(v_i,v_j)$.  The distance matrix of $G$, denoted 
$D(G)$, is the $n\times n$ symmetric matrix having its $(i,j)$-entry equal to $d(v_i,v_j )$. 
The distance matrix has attracted the attention of many researchers. The interest in this matrix 
was motivated by the connection with a communication problem (see~\cite{GrahamLovasz1978,GraphamandPollak1973} for more 
details). In an early article, Graham and Pollack \cite{GraphamandPollak1973} presented a remarkable result, proving that 
the determinant of the distance matrix of a tree $T$ on $n$ vertices only depends on $n$, being equal to 
$(-1)^{n-1}(n-1)2^{n-2}$. More recently, formulas for the determinat of connected graphs on $n$ vertices with $n$ edges
\cite{BapatKirklandandNeumann2005} (unicyclic graphs) and $n+1$ edges~\cite{Dratmanetal2021} (bicyclic graphs) have been 
computed. 

Determining the family of graphs with a given nullity for some associated matrix is a problem of interest for the 
graph-theoretic community. For instance, it is well-known that the nullity of the Laplacian matrix $L(G)$ of a given
graph $G$ coincides with the number of connected components of $G$ (see \cite{Merris1994}). Bo and Liu considered graphs 
whose adjacency matrix has rank two or three~\cite{ChengandLiu2007}; i.e., graphs with nullity $n-2$ and $n-3$, where 
$n$ is the number of vertices of the graph. Later, Cang et al. characterized graphs whose adjacency matrix has rank 
four~\cite{CHY2011} and five~\cite{CHY2012}. 

The remainder of this article is organized as follows. In Section~\ref{sec: general concepts}  we present some 
definitions and preliminary results.
Section~\ref{sec: distance rank of general graphs} is devoted to proving that for any integer $k\ge 2$, there exists a 
finite number of graphs with distance rank $k$. Section~\ref{sec: twins and null space} presents a collection of 
results in connection with the distance rank of a graph and a partition of its vertex set into sets of twins. In 
Section~\ref{sec: threshold graphs} we prove that the nullity of any threshold graph is  at most one, and we also 
present an infinite family of threshold graphs with nullity one. Finally, Section~\ref{sec: trivially perfect graphs} 
contains a sufficient condition for a trivially perfect graph to have a nonsingular distance matrix and a result that 
guarantees an example of a trivially perfect graph with nullity $\eta$, for each positive integer $\eta\ge 2$. In Section~\ref{sec: conclusions}, we close the article with some conclusions and open questions.

\section{General concepts}\label{sec: general concepts}

Let $G$ be a graph. We use $V(G)$ and $E(G)$ to denote the set of vertices of $G$ and the set of edges of $G$, respectively. We use $N_G(v)$ to denote the set of neighbors of a vertex $v\in V(G)$ and $N_G[v]=N_G(v)\cup\{v\}$, we omit the subscript in case the context is clear enough. A vertex $v$ is a \emph{universal vertex} if $N_G[v]=V(G)$. Let $S\subseteq V(G)$. We use $N_G(S)$ to denote the set of those vertices with at least one neighbor in $S$ and $N_G[S]=N_G(S)\cup S$, omitting the subscript in case the context is clear enough. Two vertices $u$ and $v$ are \emph{true twins} (resp. \emph{false twins}) if $N[u]=N[v]$ (resp. $N(u)=N(v)$). A vertex $v$ is \emph{universal} if $N[v]=V(G)$. Let $X\subseteq V(G)$. We use $G[X]$ to denote the subgraph of $G$ induced by $X$. A \emph{stable set} (or \emph{independent set}) of a graph is a set of pairwise nonadjacent vertices. By $\overline G$, we denote the \emph{complement graph} of $G$. The \emph{maximum independent number}, denoted $\alpha(G)$, is the cardinality of an independent set with the maximum number of vertices. A \emph{clique} is a set of pairwise adjacent vertices. A \emph{split graph} is a graph whose vertices can be partitioned into an
independent set and a clique. A \emph{complete graph} is a graph such that all its vertices are pairwise adjacent. We use $C_n$, $K_n$, $K_{1,n-1}$ and $P_n$ to denote the isomorphism classes of cycles, complete graphs, stars and paths, all of them on $n$ vertices, respectively. Let $\mathcal H$ be a set of graphs. A graph is said to be \emph{$\mathcal H$-free} if it does not contain any graph in $\mathcal H$ as an induced subgraph. In the case in which $\mathcal H=\{H\}$, we use \emph{$H$-free} for short. Let $G$ and $H$ be two graphs. We use $G+H$ (resp. $G\vee H$) to denote the disjoint union of $G$ and $H$ (resp. the joint between $G$ and $H$; i.e., $G+H$ plus all edges having an endpoint in $V(G)$ and the other one in $V(H)$). 

A \emph{cograph} is a $P_4$-free graph. If $G$ is a cograph, then $G$ or $\overline G$ is connected~\cite{Corneil81}. Thus, if $G$ is a connected cograph, then $G=H\vee J$, for two cographs $H$ and $J$. A graph is \emph{trivially perfect} if, for each induced subgraph, the maximum cardinality of an independent set agrees with the number of maximal cliques. Indeed, trivially perfect graphs are precisely the $\{P_4,C_4\}$-free graphs~\cite{Gol78}. In addition, a graph is trivially perfect if and only if every connected induced subgraph has a universal vertex (see~\cite{Chang96}). A graph is \emph{threshold} if it is $\{2K_2,P_4,C_4\}$-free. Observe that threshold graphs are precisely the split cographs. For more details about the graph classes described above, we refer the reader to~\cite{Golumbic2004}.

\section{Distance rank of general graphs}\label{sec: distance rank of general graphs}

The \emph{rank} of a graph $G$, denoted $\emph{rank}(G)$,	 is the rank of its adjacency matrix. For each integer $k\ge 2$ there exists an infinite family of graphs having rank $k$ (see~\cite{CHY2011}). The rank of $D(G)$, denoted $\emph{rank}_d(G)$, is called the \emph{distance rank} of $G$. Unlike what happpens with the rank of a graph, as a consequence of Ramsey's Theorem, for every integer $k\ge 2$ there exists a finite family of graphs having distance rank equal to $k$. Recall that given two integers $r,t\ge 2$ there exists a positive integer  $R(r,t)$, such that for every graph $G$ with $|V(G)|\ge R(r,t)$, $G$ contains either a clique with at least $r$ vertices or an independent set with at least $t$ vertices~\cite{Ramsey1929}. When $r=t$, $R(t)$ stands for $R(t,t)$. For bounds of $R(r,t)$ see for instance~\cite{Spencer1975}.

\subsection{General characteristic}

 Let $n\ge 2$. If $G=K_n$, clearly $n=\emph{rank}(G)=\emph{rank}_d(G)$. Besides, if $G$ is a tree on $n$ vertices, then $\emph{rank}_d(G)=n$~\cite{GraphamandPollak1973}, and thus $\emph{rank}_d(G)(K_{1,n-1})=\emph{rank}_d(G)(P_n)=n$.
Let $G$ and $H$ be two graphs. The graph $H$ is said to be an \emph{isometric subgraph} of $G$ if $H$ is a subgraph of $G$ such that $d_H(u,v)=d_G(u,v)$ for every $u,v\in V(H)$. We state the following immediate lemma without proof.

\begin{lemma}~\label{lem: isometric subgraphs}
	If $H$ is an isometric subgraph of $G$, then $\emph{rank}_d(H)\le \emph{rank}_d(G)$.
\end{lemma}
 
The diameter of a graph $G$, denoted $\emph{diam}(G)$, is the maximum distance between two vertices. An induced path $P$ of $G$ on $\emph{diam}(G)+1$ vertices is called a \emph{diameter path}. By Lemma~\ref{lem: isometric subgraphs} and~\cite{GraphamandPollak1973}, since every graph contains a diameter path as an isometric subgraph, the lemma below follows.

\begin{lemma}~\label{lem: diameter lower bound}
	If $G$ is a connected graph, then $\emph{diam}(G)+1\le \emph{rank}_d(G)$.
\end{lemma}

It is well-known that the number of vertices of a graph $G$ is upper-bounded by a function on its maximum degree $\Delta(G)$ and $\emph{diam}(G)$.

\begin{lemma}~\cite[Exercise 2.1.60]{west2001}~\label{lem: upper bound diameter and delta}
	Let $G$ be a graph. If $\emph{diam}(G)=d$ and $\Delta(G)=r$, then
	\[|V(G)|\le\frac{1+[(r-1)^d-1]r}{r-2}=f(d,r)\]
\end{lemma}

As a consequence of Ramsey's theorem we prove the main result of this section.

\begin{theorem}\label{thm: finite number of graphs with rank k}
	If $k$ is an integer with $k\ge 2$, then there is a finite number of connected graphs $G$ such that $\emph{rank}_d(G)=k$.
\end{theorem}

\begin{proof}
	Consider a connected graph $G$ such that $\emph{rank}_d(G)=k$. On the one hand if $\emph{diam}(G)\ge k$, by Lemma~\ref{lem: diameter lower bound}, $\emph{rank}_d(G)>k$. On the other hand, if $\Delta(G)\ge R(k)$, by Ramsey's Theorem, $G$ contains either a complete subgraph $K_{k+1}$ or a star $K_{1,k}$ as an isometric subgraph. Thus, Lemma~\ref{lem: isometric subgraphs}, $\emph{rank}_d(G)>k$. Hence, if $\emph{rank}_d(G)=k$, then $\emph{diam}(G)< k$ and $\Delta(G)< R(k)$. Therefore, by Lemma~\ref{lem: upper bound diameter and delta}, $|V(G)|\le f(k,R(k))$ and the result holds.
\end{proof}

\subsection{Graphs with distance rank $k\in\{2,3\}$}

A connected graph $G$ with at least three vertices contains either $P_3$ or $K_3$ as isometric subgraphs and thus $\emph{rank}_d(G)\ge 3$. For graphs used throughout this section, see Figure~\ref{fig: graphs}. In particular, it
is easy to check that $\emph{rank}_d(Pa)=\emph{rank}_d(Di)=\emph{rank}_d(Hou)=4$.

	\begin{figure}
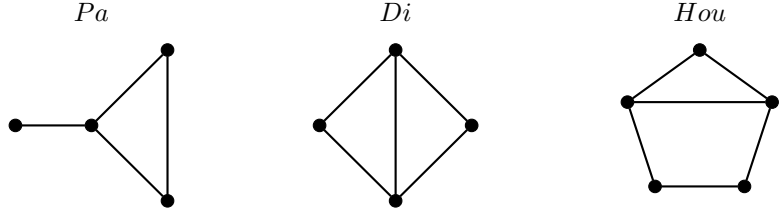

	\begin{center}
\tikz{
	\node[circle,draw,fill,scale=.5] at (-5,0) (p1) {};
	\node[circle,draw,fill,scale=.5] at (-4,0) (p2) {};
	\node[circle,draw,fill,scale=.5] at (-3,1) (p3) {};
	\node[circle,draw,fill,scale=.5] at (-3,-1) (p4) {};
	
	\draw (p1) [thick] to (p2);
	\draw (p3) [thick] to (p2);
	\draw (p4) [thick] to (p2);
	\draw (p3) [thick] to (p4);
	
	\node[circle,draw,fill,scale=.5] at (-1,0) (d1) {};
	\node[circle,draw,fill,scale=.5] at (1,0) (d2) {};
	\node[circle,draw,fill,scale=.5] at (0,1) (d3) {};
	\node[circle,draw,fill,scale=.5] at (0,-1) (d4) {};	

	\draw (d1) [thick] to (d3);
	\draw (d1) [thick] to (d4);
	\draw (d2) [thick] to (d3);
	\draw (d2) [thick] to (d4);
	\draw (d4) [thick] to (d3);
	
\foreach \i in {0,1,...,4}
{
\node[shape=circle,draw,fill,scale=.5] at ($(90-72*\i:1)+(4,0)$) (h\i) {};
}
\foreach \i in {0,1,...,4}{
\pgfmathtruncatemacro{\j}{(mod(\i+1,5))};
	\draw (h\i) [thick] to (h\j);
}
	\draw (h1) [thick] to (h4);

	\node[circle] at (-4,1.5) (Pa) {$Pa$};
	\node[circle] at (0,1.5) (Di) {$Di$};
	\node[circle] at (4,1.5) (Ho) {$Hou$};
}		
	\caption{$Pa$, the paw graph; $Di$, the diamond graph; and $Hou$, the house graph.}\label{fig: graphs}
    \end{center}
    \end{figure}

\begin{remark}
	A connected graph $G$ has $\emph{rank}_d(G)=2$ if and only if $G=K_2$.
\end{remark}

The following lemma is a consequence of the isometric subgraph definition.

\begin{lemma}\label{lem: distance two}
	If $H$ is a connected induced subgraph of a connected graph $G$ such that $d_H(u,v)\le 2$ for every $u,v\in V(H)$, then $H$ is an isometric subgraph of $G$.
\end{lemma}

As a consequence of the above lemma the graphs with distance rank equals three are cographs.

\begin{lemma}\label{lem: distance rank three implies cograph}
	If $G$ is a connected graph with $\emph{rank}_d(G)=3$, then $G$ is a cograph.
\end{lemma}

\begin{proof}
We prove the contrapositive statement. Assume that $G$ contains a path  with four vertices $P:\;a,b,c,d$ as an induced subgraph. If $P$ was an isometric subgraph, then $\emph{rank}_d(G)\ge 4$ by Lemma~\ref{lem: isometric subgraphs}. Assume that $d_G(a,d)=2$. Consequently, there exists a vertex $v$ in $G$ that is adjacent to $a$ and $d$. Thus $G[\{a,b,c,d,v\}]$ contains a diamond as an induced subgraph or is isomorphic to $C_5$ or the house. Since the diamond and the house have distance rank $4$ and the $C_5$ has distance rank $5$, it follows from Lemma~\ref{lem: distance two}  that $\emph{rank}_d(G)\ge 4$. Thus, if $G$ is not a cograph, then $\emph{rank}_d(G)\geq 4$. Therefore,
the result follows.
\end{proof}

\begin{theorem}
	If $G$ is a connected graph with $\emph{rank}_d(G)=3$, then $G$ is one of the following graphs: $K_3$, $P_3$, or $C_4$.
\end{theorem}

\begin{proof}
	Let $G$ be a graph with $\emph{rank}_d(G)=3$. By Lemma~\ref{lem: distance rank three implies cograph} $G$ is a cograph. As $G$ is also connected and has at least $3$ vertices, we have $G=F\lor H$, where $F$ and $H$ are two
	non-empty cographs. Notice that, by Lemma~\ref{lem: distance two}, $G$ does not contain a paw as an induced subgraph because the distance rank of the paw is equal to $4$. Since $G$ contains neither a diamond nor a paw as induced subgraphs, $H$ (resp. $F$) contains neither $P_3$ nor $K_2+P_1$ as induced subgraphs. Hence $H$ (resp. $F$) is either a complete graph or isomporphic to $nK_1$. Assume first that one of $H$ and $F$ is a complete graph with at least two vertices, say $H$. By Lemma~\ref{lem: distance two}, since $\emph{rank}_d(K_4)=4$, $H$ has exactly two vertices. Since $G$ contains neither a diamond nor $K_4$ as induced subgraphs, $F$ contains only one vertex, and thus $G$ is isomorphic to $K_3$. We can assume now that $F$ and $H$ are isomorphic to $rK_1$ and $sK_1$, respectively. Since $G$ does not contain $S_{1,3}$ as an induced subgraph, we conclude that $r\le 2$ and $s\le 2$. Therefore, $G$ is isomorphic to $P_3$, or $C_4$. 
\end{proof}

\section{Twins and null space}\label{sec: twins and null space}

Let $G$ be a graph with vertices $v_1,v_2,\ldots, v_n$, and assume that $v_1$ and $v_2$ are either true twins or false twins. Notice that if $j\not\in \{1,2\}$, then $d_G(v_1,v_j)=d_G(v_2,v_j)$. Let $D$ be the distance matrix of $G$
and $\vec{x}$ a vector in the null space of $D$. We denote the coordinate of $\vec{x}$ that corresponds to  vertex $v_i$ as $\vec{x}_{v_i}$. Notice that the coordinate corresponding to $v_i$ of $D\vec{x}$ satisfies
\[
[D\vec{x}]_{v_i}=\sum_{j=1}^nd_G(v_i,v_j)\vec{x}_{v_j},
\]
for every $1\le i\le n$. Hence
\begin{align*}
[D\vec{x}]_{v_1}-[D\vec{x}]_{v_2}=&\sum_{j=1}^nd_G(v_1,v_j)\vec{x}_{v_j}-\sum_{j=1}^nd_G(v_2,v_j)\vec{x}_{v_j}\\
=&d_G(v_1,v_1)\vec{x}_{v_1}+d_G(v_1,v_2)\vec{x}_{v_2}-d_G(v_1,v_2)\vec{x}_{v_1}-d_G(v_2,v_2)\vec{x}_{v_2}\\
=&d_G(v_1,v_2)(\vec{x}_{v_2}-\vec{x}_{v_1}).
\end{align*}
Since $\vec{x}$ is in the null space of $D$, $[D\vec{x}]_{v_1}=[D\vec{x}]_{v_2}=0$. Thus $d_G(v_1,v_2)(\vec{x}_{v_2}-\vec{x}_{v_1})=0$,
which implies $\vec{x}_{v_2}=\vec{x}_{v_1}$. From the preceding discussion we obtain the following result.

\begin{lemma}\label{lem:twinsiguales}
Let $G$ be a graph with distance matrix $D$. If $v_i$ and $v_j$ are either true twins or false twins and 
$\vec{x}$ is in the null space of $D$, then $\vec{x}_{v_i}=\vec{x}_{v_j}$.
\end{lemma}

Lemma \ref{lem:twinsiguales} allows to use a smaller matrix to study the null space of $D$.
To do that, we introduce some notation. We say that a partition  $\mathcal{W}=\{W_1,W_2,\ldots,W_k\}$  of the set of vertices
is a \textit{twin partition} of a graph $G$ if $W_i$ is either a set of true twins or a set of false twins
for every $i$. Notice that we allow $|W_i|=1$. If $W_i$ is a set of true (false) twins for every $i$, then
we say that $\mathcal{W}$ is a \textit{true (false) twin partition} of $G$.

Let $\mathcal{W}=\{W_1,W_2,\ldots,W_k\}$ be a twin partition of $G$ and $w_1,\ldots, w_k$ a set
of vertices with $w_i\in W_i$ for each $1\le i\le k$. We define the \textit{quotient matrix} $D/\mathcal{W}$
by
\[
(D/\mathcal{W})_{i,j}=
\begin{cases}
|W_j|d_G(w_i,w_j)& \text{ if $i\neq j$,}\\
(|W_i|-1)&\text{ if $i=j$ and $W_i$ is a set of true twins,}\\
2(|W_i|-1)&\text{ if $i=j$ and $W_i$ is a set of false twins.}\\
\end{cases}
\]

Let $\vec{x}\in \mathbb{R}^n$ be a vector such that $\vec{x}_{v_i}=\vec{x}_{v_j}$ if $v_i$ and $v_j$ are twin vertices and
let $\vec{y}\in\mathbb{R}^k$ such that $\vec{y}_{v_i}=\vec{x}_{w_i}$.
We have
\begin{align*}
[D/\mathcal{W} \vec{y}]_{v_i}=\sum_{j=1,j\neq i}^{k}d_G(w_i,w_j)|W_j|\vec{x}_{w_j} +c_i(|W_i|-1)\vec{x}_{w_i},
\end{align*} 
where $c_i=1$ if $W_i$ consists of true twins and $c_i=2$ if $W_i$ consists of false twins.
On the other hand
\begin{align*}
[D\vec{x}]_{w_i}=&\sum_{v_j\in V}d_G(v_i,v_j)\vec{x}_{v_j}\\
=&\sum_{\ell=1,\ell \neq i}^k\sum_{v_j\in W_\ell}d_G(v_i,v_j)\vec{x}_{v_j}+\sum_{v_j\in W_i, v_j\neq w_i}d_G(v_i,v_j)\vec{x}_{v_j}\\
=&\sum_{\ell=1,\ell\neq i}^k|W_\ell|d_G(v_i,w_\ell)\vec{x}_{w_\ell}+c_i(|W_i|-1)\vec{x}_{w_i}\\
=&[D/\mathcal{W} \vec{y}]_{v_i}.
\end{align*} 
Thus, $\vec{x}$ is in the null space of $D$ if and only if $\vec{y}$ is in the null space of $D/\mathcal{W}$.
Combined with Lemma \ref{lem:twinsiguales}, this implies that the nullity of $D$ equals the nullity of $D/\mathcal{W}$.
\begin{lemma}\label{lem:matrizcociente}
Let $G$ be a graph, $D$ the distance matrix of $G$ and $\mathcal{W}=\{W_1\ldots,$ $W_k\}$ 
be a partition of the vertices of $G$ into sets of twins, each of them consisting of either true twins 
or false twins. For each $i$, let $w_i$ be a vertex in $W_i$. If $D/\mathcal{W}$ is the matrix defined as
\[
D/\mathcal{W}_{i,j}=\begin{cases} 
|W_j|d_G(w_i,w_j)&\text{if $i\neq j$,}\\
|W_i|-1&\text{if $i=j$ and $W_i$ consists of true twins, and}\\
2(|W_i|-1)&\text{if $i=j$ and $W_i$ consists of false twins,}\\
\end{cases}
\]
then the nullity of $D$ is equal to the nullity of $D/\mathcal{W}$.
\end{lemma}

\section{Threshold graphs}~\label{sec: threshold graphs}
It is well-known that we can obtain any threshold graph by repeatedly adding either isolated vertices or universal vertices. Thus, a threshold graph can be represented by 
a finite sequence  $(a_i)_{i=1}^n$, with $a_i\in\{0,1\}$, with edges of the form $\{v_i,v_j\}$ if $a_i=1$ and $i>j$.
We are going to assume $a_n=1$ as otherwise the graph is not connected. Notice that $a_1$ can be assumed to be $0$ since otherwise would give place to the same graph. Since the sequence $(a_i)$ consists of some consecutive zeros, followed by 
consecutive ones and so on, we can write it as 
$[0^{n_1},1^{n_2},0^{n_3},\ldots ,1^{n_{2k-2}},0^{n_{2k-1}},1^{n_{2k}}]$, 
where $a^b$ represents $b$ consecutive copies of the number $a$. Notice that in 
$[0^{n_1},1^{n_2},0^{n_3},\ldots ,1^{n_{2k-2}},0^{n_{2k-1}},1^{n_{2k}}]$ the number $0$ 
appears in every odd position and $1$ in every even position, thus the only values providing information are 
$(n_i)$. We can represent $(a_i)$ with the sequence $[n_1,n_2,n_3,\ldots,n_{2k-2},n_{2k-1},n_{2k}]$,   
called the \emph{power sequence} of the threshold graph $G$.

As every $0$ vertex is at distance $2$ of all previous vertices and every $1$ vertex is at distance $1$ of all previous
vertices, if  $[n_1,n_2,n_3,\ldots,n_{2k-2},n_{2k-1},n_{2k}]$ is the power sequence of a threshold graph $G$, then 
the distance matrix $D$ of $G$ is
\[
\begin{pmatrix}
    2(J-I) & J & 2J & J & \ldots & J & 2J & J \\
    J & J-I & 2J & J & \ldots &J & 2J &J \\
    2J & 2J & 2(J-I) & J &\ldots &J & 2J &J \\
    J & J & J & J-I & \ldots & J & 2J & J\\
    \vdots & \vdots & \vdots & \vdots & \ldots & \vdots & \vdots & \vdots\\
    J & J & J & J & \ldots &J-I & 2J & J\\
    2J & 2J & 2J & 2J & \ldots & 2J & 2(J-I) & J\\
    J & J & J & J & \ldots & J & J & (J-I)
    \end{pmatrix},
    \]
where each $J$ in position $i,j$ stands for a block of $n_i\times n_j$ ones, and each $I$ in position $i,i$ an 
$n_i\times n_i$ identity matrix.
Notice that consecutive zeros produce false twins, whereas consecutive ones produce true twins. 
We can partition the vertices of $G$ into $\mathcal{W}=\{W_1,\ldots, W_{2k}\}$, where $W_i$ consists of $n_i$ 
false twins if $i$ is odd and $n_i$ true twins if $i$ is even. Consequently $D/ \mathcal W$ equals
\[
\begin{pmatrix}
2n_1-2&n_2& 2n_3 & n_4 & \ldots & n_{2k-2} & 2n_{2k-1} & n_{2k}\\
n_1&n_2-1& 2n_3 & n_4 & \ldots & n_{2k-2} & 2n_{2k-1} & n_{2k}\\
2n_1&2n_2& 2n_3-2 & n_4 & \ldots & n_{2k-2} & 2n_{2k-1} & n_{2k}\\
n_1&n_2& n_3 & n_4-1 & \ldots & n_{2k-2} & 2n_{2k-1} & n_{2k}\\
\vdots & \vdots & \vdots & \vdots & \ldots & \vdots & \vdots & \vdots\\
n_1&n_2& n_3 & n_4 & \ldots & n_{2k-2}-1 & 2n_{2k-1} & n_{2k}\\
2n_1&2n_2& 2n_3 & 2n_4 & \ldots & 2n_{2k-2} & 2n_{2k-1}-2 & n_{2k}\\
n_1&n_2& n_3 & n_4 & \ldots & n_{2k-2} & n_{2k-1} & n_{2k}-1\\
\end{pmatrix},
\]
Lemma \ref{lem:matrizcociente} allows us to use $D/ \mathcal W$ instead of $D$ to study its nullity. Given a matrix $A$ having $m$ rows, we denote by $r_i(A)$ the $i$-th row of $A$ for each $1\le i\le m$. When the context is clear enough, we use $r_i$ for shortness. 
We proceed to apply row operations to $D/\mathcal{W}$. We begin by doing $r_i-r_{i+1} \to r_i$ for $i$ moving from $1$ to $2k-1$
\[
\begin{pmatrix}
n_1-2&1& 0 & 0 & \ldots & 0  & 0 & 0\\
-n_1&-n_2-1& 2 & 0 & \ldots & 0 & 0 & 0\\
n_1&n_2& n_3-2 & 1 & \ldots & 0 & 0 & 0\\
-n_1&-n_2& -n_3 & -n_4-1 & \ldots & 0 & 0 & 0\\
\vdots & \vdots & \vdots & \vdots & \ldots & \vdots & \vdots & \vdots\\
-n_1&-n_2& -n_3 & -n_4 & \ldots & -n_{2k-2}-1 & 2 & 0\\
n_1&n_2& n_3 & n_4 & \ldots & n_{2k-2} & n_{2k-1}-2 & 1\\
n_1&n_2& n_3 & n_4 & \ldots & n_{2k-2} & n_{2k-1} & n_{2k}-1\\
\end{pmatrix},
\]
we multiply every even row by $-1$, but the last one
\[
\begin{pmatrix}
n_1-2&1& 0 & 0 & \ldots & 0  & 0 & 0\\
n_1&n_2+1& -2 & 0 & \ldots & 0 & 0 & 0\\
n_1&n_2& n_3-2 & 1 & \ldots & 0 & 0 & 0\\
n_1&n_2& n_3 & n_4+1 & \ldots & 0 & 0 & 0\\
\vdots & \vdots & \vdots & \vdots & \ldots & \vdots & \vdots & \vdots\\
n_1&n_2& n_3 & n_4 & \ldots & n_{2k-2}+1 & -2 & 0\\
n_1&n_2& n_3 & n_4 & \ldots & n_{2k-2} & n_{2k-1}-2 & 1\\
n_1&n_2& n_3 & n_4 & \ldots & n_{2k-2} & n_{2k-1} & n_{2k}-1\\
\end{pmatrix}.
\]
Finally, we do $r_{2k-i}-r_{2k-i-1} \to r_{2k-i}$ for $i$ moving from $0$ to $2k-2$,
\[
\begin{pmatrix}
n_1-2&1& 0 & 0 & \ldots & 0  & 0 & 0\\
2&n_2& -2 & 0 & \ldots & 0 & 0 & 0\\
0&-1& n_3 & 1 & \ldots & 0 & 0 & 0\\
0&0& 2 & n_4 & \ldots & 0 & 0 & 0\\
\vdots & \vdots & \vdots & \vdots & \ldots & \vdots & \vdots & \vdots\\
0&0&0&0& \ldots & n_{2k-2} & -2 & 0\\
0&0&0&0& \ldots & -1 & n_{2k-1} & 1\\
0&0&0&0& \ldots & 0 & 2 & n_{2k}-2\\
\end{pmatrix}.
\]
The first $2k-1$ rows are linearly independent. Thus
the nullity of $D/\mathcal{W}$ is at most $1$. Lemma \ref{lem:matrizcociente}
yields the following.
\begin{theorem}\label{thm: nulity of threshold graphs}
If $D$ is the distance matrix of a connected threshold graph, then the nullity of $D$ 
is at most $1$.
\end{theorem}
We now want to find precisely which threshold graphs have nullity $1$. 
Dividing even rows of the last matrix by $-2$, we obtain
\[
\begin{pmatrix}
n_1-2&1& 0 & 0 & \ldots & 0  & 0 & 0\\
-1&-n_2/2& 1 & 0 & \ldots & 0 & 0 & 0\\
0&-1& n_3 & 1 & \ldots & 0 & 0 & 0\\
0&0& -1 & -n_4/2 & \ldots & 0 & 0 & 0\\
\vdots & \vdots & \vdots & \vdots & \ldots & \vdots & \vdots & \vdots\\
0&0&0&0& \ldots & -n_{2k-2}/2 & 1 & 0\\
0&0&0&0& \ldots & -1 & n_{2k-1} & 1\\
0&0&0&0& \ldots & 0 & -1 & (2-n_{2k})/2\\
\end{pmatrix},
\]
that has the same nullity as $D/\mathcal{W}$.
Notice that if we let 
\[
\alpha_i=\begin{cases}
n_1-2&\text{if $i=1$}\\
n_i&\text{if $i>1$ is odd}\\
-n_i/2&\text{if $i<2k$ is even}\\
(2-n_{2k})/2&\text{if $i=2k$}
\end{cases}
\]
the last matrix is of the form
\[
\begin{pmatrix}
\alpha_1&1& 0 & 0 & \ldots & 0  & 0 & 0\\
-1&\alpha_2& 1 & 0 & \ldots & 0 & 0 & 0\\
0&-1&\alpha_3& 1 & \ldots & 0 & 0 & 0\\
0&0& -1 & \alpha_4 & \ldots & 0 & 0 & 0\\
\vdots & \vdots & \vdots & \vdots & \ldots & \vdots & \vdots & \vdots\\
0&0&0&0& \ldots & \alpha_{2k-2} & 1 & 0\\
0&0&0&0& \ldots & -1 & \alpha_{2k-1} & 1\\
0&0&0&0& \ldots & 0 & -1 & \alpha_{2k}\\
\end{pmatrix}.
\]
We can obtain the determinant of this last matrix inductively. Let $D_i$ be the main minor of $D$ obtained by deleting each row $k$ greater than $i$ and its corresponding columns, and let $d_i$ be the determinant
of $D_i$. It is not hard to prove that
$d_1=\alpha_1$ and $d_2=1+\alpha_1\alpha_2$ and
\begin{equation*}
d_i=\alpha_id_{i-1}+d_{i-2},
\end{equation*}
for each integer $3\le i\le 2k$. 

Thanks to the recursion, we can find some infinite families of threshold graphs with nullity $1$.
For example, if $\alpha_1,\ldots,\alpha_{2k-2}$ are such that $d_{2k-2}=0$, then $\alpha_{2k}=0$ implies $d_{2k}=0$ regardless of the value of $\alpha_{2k-1}$. As a way to apply this, notice that both $(\alpha_1,\alpha_2)=(2,-1/2)$
and $(\alpha_1,\alpha_2)=(1,-1)$ imply $d_2=0$. In addition, if $\alpha_4=0$, then $[4,1,n_3,2]$ and $[3,2,n_3,2]$ are power sequences of threshold graphs with distance nullity $1$ for every $n_3$, meaning that $K_2\vee (n_3K_1+(K_1\vee 4K_1))$ and $K_2\vee(n_3K_1+(K_2\vee 3K_1))$ are threshold graph whose distance matrices have nullity one.

Unfortunately if we wanted to keep applying this construction as is to yield a power sequence of length $6$ we would need to do $[4,1,n_3,0,n_5,2]=[4,1,n_3+n_5,2]$ because of the difference between $\alpha_i$ when $i<2k$ and $\alpha_{2k}$. 
What we can do instead is use the fact that, when $d_{i-2}=0$, we have
\begin{align*}
d_i&=\alpha_id_{i-1}\\
d_{i+1}&=\alpha_{i+1}d_i+d_{i-1}=(\alpha_{i+1}\alpha_i+1)d_{i-1}\\
\end{align*}
which is similar to how the recursion begins, multiplying by $d_{i-1}$ and replacing $(\alpha_1,\alpha_2)$ with $(\alpha_i,\alpha_{i+1})$. Thus, if $\alpha_1,\ldots,\alpha_{i}$ yield $d_{i}=0$ and $\bar{\alpha_1},\ldots,\bar{\alpha_j}$ imply $\bar{d_j}=0$, setting $\alpha_{i+k}=\bar{\alpha_k}$ implies $d_{i+j}=0$. As a way to apply this,
we can use $(\alpha_1,\alpha_2)=(1,-1)$ together with $(\bar{\alpha_1},\bar{\alpha_2},\bar{\alpha_3},\bar{\alpha_4})=(1,-1,\epsilon,0)$, with $\epsilon$ being any value we want to choose. 
This yields that threshold graphs with power sequences $[3,2,1,2,\epsilon,2]$ have nullity $1$. And repeatedly applying this construction,  we get that threshold graphs with power sequences of the form
\[
[3,2,1,2,1,2,1,2,1,2,1,2,\ldots,1,2,\epsilon,2]
\]
have nullity $1$.

\section{Trivially perfect graphs}~\label{sec: trivially perfect graphs}
In this section, we give sufficient conditions for a trivially perfect graph to have a nonsingular distance matrix. Let $G$ be a trivially perfect graph  and let $\mathcal{W}$ be a true twin partition of $G$. There exists a tree $T=(\mathcal W, E)$, called \emph{rooted clique tree of $G$}, such that if $W,W'\in \mathcal{W}$, $w\in W$ and $w'\in W'$, then $w$ and $w'$ are adjacent if and only if $W=W'$, or $W'$ is descendant of $W$ in $T$ or vice versa. By $T_W$, we denote the subtree of $T$ rooted at $W$ containing all descendants of $W$. 
The \emph{arrow matrix of $T$} is recursively defined  as follows. If $\mathcal{W}=\{R\}$, $A_T=|R|-1$. Assume that $|W|\ge 3$.

Let the elements of $\mathcal{W}$ be numbered as follows:
\begin{itemize}
\item if $i<j$ then
$W_i$ is not a descendant of $W_j$;
\item if $i<j<k$ and $W_k$ is a descendant of $W_i$, then $W_j$ is a descendant of $W_i$.
\end{itemize}
See Fig.~\ref{fig: tp graph}. Further, let $W_{h_1},W_{h_2},\ldots,W_{h_\ell}$ be the children of $R=W_1$, renumbered so that
if $i<j$, $W_i=W_{h_m}$ and $W_j=W_{h_n}$, then $h_m<h_n$. We define the \emph{arrow matrix} of $T$ as 
\[A_T=
\begin{pmatrix}
|R|+1&|W_2|&|W_3|&\cdots&|W_{\ell}|\\
|R|\\
|R|\\
\vdots&&\text{\huge $B_T$}\\
|R|\\
|R|\\
\end{pmatrix}
,\]
where
\[
B_T=\begin{pmatrix}
A_1&\0&\cdots& \0\\
\0 &A_2&\cdots&\0\\
\vdots&\vdots&\ddots&\vdots\\
\0&\0&\cdots&A_{\ell}\\
\end{pmatrix}
\]
where $A_i$ is the arrow matrix of $T(V_i)$. The ordering of $\mathcal W$ induced by the rows of $A_T$ is called an \emph{arrow ordering}.

\begin{theorem}\label{thm: nonsingular trivially perfect graphs}
	Let $G$ be a trivially perfect graph, having a true twin partition $\mathcal{W}=\{W_1,W_2,\ldots,W_k\}$ such that $|W_i|\ge 6$ for each $i=1,\ldots,k$, then $D(G)$ has an inverse, i.e.; $\eta(G)=0$. 
\end{theorem}
As the proof of Theorem~\ref{thm: nonsingular trivially perfect graphs} is a bit technical, we give an illustration
of how it works before proceeding with the actual proof.
\subsection{Illustration of Theorem~\ref{thm: nonsingular trivially perfect graphs}}~\label{subsec: example}

Consider the trivially perfect graph $G=K_6\vee((K_7\vee(K_9+K_8))+(K_9\vee((K_8\vee(K_6+K_7))+K_6)))$ with the vertex set partition $\mathcal W=\{W_1,W_2,W_3,W_4$, $W_5,W_6\}$ (see Fig.~\ref{fig. trivially perfect example}),  whose rooted
clique tree appears on Figure \ref{fig: tp graph}. Notice that the quotient matrix $D/\mathcal W$ is

\[
\begin{pmatrix}
	6-1& 7  & 9&  8&  9&  8&  6&  7&  6\\
	6  & 7-1& 9&  8& 2\cdot 9& 2\cdot 8& 2\cdot 6& 2 \cdot 7& 2 \cdot 6\\
	6&  7 &  9-1& 2\cdot 8& 2\cdot 9& 2\cdot 8& 2\cdot 6& 2 \cdot 7& 2 \cdot 6\\
	6&  7& 2\cdot 9&  8-1& 2\cdot 9& 2\cdot 8& 2\cdot 6& 2 \cdot 7& 2 \cdot 6\\
	6& 2\cdot 7& 2\cdot 9& 2\cdot 8&  9-1&  8&  6&  7&  6\\
	6& 2\cdot 7& 2\cdot 9& 2\cdot 8&  9&  8-1&  6&  7& 2\cdot 6\\
	6& 2\cdot 7& 2\cdot 9& 2\cdot 8&  9&  8&  5& 2\cdot 7& 2\cdot 6\\
	6& 2\cdot 7& 2\cdot 9& 2\cdot 8&  9&  8& 2\cdot 6&  7-1& 2\cdot 6\\
	6& 2\cdot 7& 2\cdot 9& 2\cdot 8&  9& 2\cdot 8& 2\cdot 6& 2\cdot 7&  6-1\\
\end{pmatrix},
\]

where the $i$-th row represents $W_i$. We denote such a row by $r_{W_i}$.
\begin{figure}
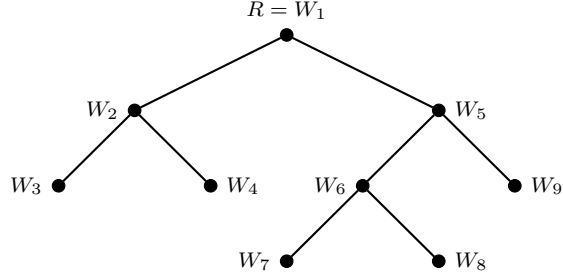
\label{fig. trivially perfect example}
	\begin{center}
\tikz{
	\node[circle,draw,fill,scale=.5,label=above:{\footnotesize $R=W_1$}] at (0,0) (r) {};
	\node[circle,draw,fill,scale=.5,label=left:{\footnotesize $W_2$}] at (-2,-1) (h1) {};
	\node[circle,draw,fill,scale=.5,label=right:{\footnotesize $W_5$}] at (2,-1) (h2) {};
	\node[circle,draw,fill,scale=.5,label=left:{\footnotesize $W_3$}] at (-3,-2) (h11) {};
	\node[circle,draw,fill,scale=.5,label=right:{\footnotesize $W_4$}] at (-1,-2) (h12) {};
	\node[circle,draw,fill,scale=.5,label=left:{\footnotesize $W_6$}] at (1,-2) (h21) {};
	\node[circle,draw,fill,scale=.5,label=right:{\footnotesize $W_9$}] at (3,-2) (h22) {};	
	\node[circle,draw,fill,scale=.5,label=left:{\footnotesize $W_7$}] at (0,-3) (h211) {};	
	\node[circle,draw,fill,scale=.5,label=right:{\footnotesize $W_8$}] at (2,-3) (h212) {};	

	\draw (r) [thick] to (h1);
	\draw (r) [thick] to (h2);
	\draw (h1) [thick] to (h11);
	\draw (h1) [thick] to (h12);
	\draw (h2) [thick] to (h21);
	\draw (h2) [thick] to (h22);
	\draw (h21) [thick] to (h211);
	\draw (h21) [thick] to (h212);
}
	\end{center}
	\caption{Rooted tree of $G=K_6\vee((K_6\vee(K_6+K_6))+(K_6\vee((K_6\vee(K_6+K_6))+K_6)))$; i.e., each $W_i$ is a clique with six vertices. Rooted tree of $G=K_6\vee((K_7\vee(K_9+K_8))+(K_9\vee((K_8\vee(K_6+K_7))+K_6)))$.}\label{fig: tp graph}
\end{figure}

Now  we apply on $D/\mathcal W$ the following elementary operations, first $r_{W_i}-2r_{W_1}\to r_{W_i}$  and then $-r_{W_i}\to r_{W_i}$, for each $i\geq 2$, obtaining the following matrix

\[M_1=
\begin{pmatrix}
5 & 7 & 9 & 8 & 9 & 8& 6& 7& 6\\
4 & 8 & 9 & 8 & 0 & 0& 0& 0& 0\\
4 & 7 & 10& 0 & 0 & 0& 0& 0& 0\\
4 & 7 & 0 & 9 & 0 & 0& 0& 0& 0\\
4 & 0 & 0 & 0 & 10& 8& 6& 7& 6\\
4 & 0 & 0 & 0 & 9 & 9& 6& 7& 0\\
4 & 0 & 0 & 0 & 9 & 8& 7& 0& 0\\
4 & 0 & 0 & 0 & 9 & 8& 0& 8& 0\\
4 & 0 & 0 & 0 & 9 & 0& 0& 0& 7\\
\end{pmatrix}.
\]

To make some more $0$'s we do the following row operations. First we subtract from the row corresponding to the root 
the rows corresponding to its children, i.e., $r_{W_1}-r_{W_2}-r_{W_5}\to r_{W_1}$. Do the same for
$r_{W_5}$, $r_{W_5}-r_{W_6}-r_{W_9}\to r_{W_5}$. This was done because $W_5$ has grandchildren (i.e. it has child who
has children of its own). This yields the matrix

\[M_2'=
\begin{pmatrix}
-3& -1& 0 & 0 & -1& 0& 0& 0& 0\\
4 & 8 & 9 & 8 & 0 & 0& 0& 0& 0\\
4 & 7 & 10& 0 & 0 & 0& 0& 0& 0\\
4 & 7 & 0 & 9 & 0 & 0& 0& 0& 0\\
-4& 0 & 0 & 0 & -8&-1& 0& 0&-1\\
4 & 0 & 0 & 0 & 9 & 9& 6& 7& 0\\
4 & 0 & 0 & 0 & 9 & 8& 7& 0& 0\\
4 & 0 & 0 & 0 & 9 & 8& 0& 8& 0\\
4 & 0 & 0 & 0 & 9 & 0& 0& 0& 7\\
\end{pmatrix}.
\]

We keep making $0$'s appear as follows. We take
every vertex that has children, but not grandchildren, and use them to make $0$'s. This means we do  $r_{W_2}-\frac{9}{10}r_{W_3}-\frac{8}{9}r_{W_4} \to r_{W_2} $ and $r_{W_6}-\frac{6}{7}r_{W_7}-\frac{7}{8}r_{W_8} \to r_{W_6}$. This gives the matrix

\[M_2=
\begin{pmatrix}
-3& -1& 0 & 0 & -1& 0& 0& 0& 0\\
\frac{-142}{45} & \frac{-407}{90} & 0 & 0 & 0 & 0& 0& 0& 0\\
4 & 7 & 10& 0 & 0 & 0& 0& 0& 0\\
4 & 7 & 0 & 9 & 0 & 0& 0& 0& 0\\
-4& 0 & 0 & 0 & -8&-1& 0& 0&-1\\
\frac{-41}{14} & 0 & 0 & 0 & \frac{-369}{56} & \frac{-34}{7}& 0& 0& 0\\
4 & 0 & 0 & 0 & 9 & 8& 7& 0& 0\\
4 & 0 & 0 & 0 & 9 & 8& 0& 8& 0\\
4 & 0 & 0 & 0 & 9 & 0& 0& 0& 7\\
\end{pmatrix}.
\]

We can do now something similar for $r_{W_5}$, although
we need to multiply $r_{W_6}$ by a different value. We do $r_{W_5}-\frac{7}{34}r_{W_6}\to r_{W_5}$ and then $r_{W_5} +\frac{1}{7}r_{W_9} \to r_{W_{5}}$.
Notice that in this case we have $\frac{7}{34}=\frac{34}{7}^{-1}=\frac{-1}{(M_2)_{6,6}}$, and $\frac{1}{7}=\frac{1}{|W_9|+1}$. This yields the matrix

\[N=
\begin{pmatrix}
-3& -1& 0 & 0 & -1& 0& 0& 0& 0\\
\frac{-142}{45} & \frac{-407}{90} & 0 & 0 & 0 & 0& 0& 0& 0\\
4 & 7 & 10& 0 & 0 & 0& 0& 0& 0\\
4 & 7 & 0 & 9 & 0 & 0& 0& 0& 0\\
\frac{-1345}{476}& 0 & 0 & 0 & \frac{-10201}{1904}& 0& 0& 0& 0\\
\frac{-41}{14} & 0 & 0 & 0 & \frac{-369}{56} & \frac{-34}{7}& 0& 0& 0\\
4 & 0 & 0 & 0 & 9 & 8& 7& 0& 0\\
4 & 0 & 0 & 0 & 9 & 8& 0& 8& 0\\
4 & 0 & 0 & 0 & 9 & 0& 0& 0& 7\\
\end{pmatrix}.
\]

Finally, we can do the same process for $r_{W_1}$, using $r_{W_2}$ and $r_{W_5}$. Thus we do  $r_{W_1}-\frac{90}{407}r_{W_2}\to r_{W_1}$ and then $r_{W_1}-\frac{1904}{10201}r_{W_5} \to r_{W_1}$. In this case, as neither $W_2$ nor $W_5$ were leaves, we are just using $\frac{-90}{407}=N_{2,2}^{-1}$ and $\frac{-1904}{10201}=N_{5,5}^{-1}$. Thus, we obtain the following lower triangular matrix, which is non-singular because it does not have any zeros in the main diagonal.

\[
\begin{pmatrix}
 \frac{-7368677}{4151807} & 0& 0 & 0 & 0& 0& 0& 0& 0\\[0.75ex]
\frac{-142}{45} & \frac{-407}{90} & 0 & 0 & 0 & 0& 0& 0& 0\\[0.75ex]
4 & 7 & 10& 0 & 0 & 0& 0& 0& 0\\[0.75ex]
4 & 7 & 0 & 9 & 0 & 0& 0& 0& 0\\[0.75ex]
\frac{-1345}{476}& 0 & 0 & 0 & \frac{-10201}{1904}& 0& 0& 0& 0\\[0.75ex]
\frac{-41}{14} & 0 & 0 & 0 & \frac{-369}{56} & \frac{-34}{7}& 0& 0& 0\\[0.75ex]
4 & 0 & 0 & 0 & 9 & 8& 7& 0& 0\\[0.75ex]
4 & 0 & 0 & 0 & 9 & 8& 0& 8& 0\\[0.75ex]
4 & 0 & 0 & 0 & 9 & 0& 0& 0& 7\\[0.75ex]
\end{pmatrix}.
\]

\subsection{Proof of Theorem~\ref{thm: nonsingular trivially perfect graphs}}\label{subsec: proof}
Before proceeding with the proof, we need to define the height of the vertices of a rooted tree.
This definition is done inductively. If $v$ has no children we define the \emph{height of $v$}
as $h(v)=0$. If $v$ has children, and the height of every child of $v$ has been defined, we define the \emph{height of $v$} as 
\[
h(v)=1+\max_{w|\text{$w$ is a child of $v$}}{h(w)}.
\]
Thus, for the vertices of the rooted tree in Figure~\ref{fig: tp graph} we have
\begin{align*}
h(W_3)=&h(W_4)=h(W_7)=h(W_8)=h(W_9)=0,\\
h(W_2)=&h(W_6)=1,\\
h(W_5)=&2,\\
h(W_1)=&3.
\end{align*}
We are ready now to present the prove Theorem~\ref{thm: nonsingular trivially perfect graphs}.
\begin{proof}[Proof of Theorem~\ref{thm: nonsingular trivially perfect graphs}]
	 Let $T=(\mathcal W, E)$ be a rooted clique tree of $G$. Consider now an arrow ordering $W_1,\ldots,W_{|\mathcal W}|$. The quotient matrix of $G$, under this ordering, has the following structure.
	 
	 \[D/\mathcal{W}=
	 \begin{pmatrix}
	 |R|-1      &\vec{x}_1^t& \vec{x}_2^t&\cdots & \vec{x}_k^t\\
	 |R|\cdot\1 & B_1 & 2\cdot\1\vec{x}_2^t&\cdots& 2\cdot\1\vec{x}_k^t\\
	 |R|\cdot\1 & 2\cdot\1 \vec{x}_1^t&B_2&\cdots& 2\cdot\1\vec{x}_k^t\\
	 \vdots     & \vdots              &\vdots&\ddots&\vdots\\
	 |R|\cdot\1 & 2\cdot\1 \vec{x}_1^t& 2\vec{x}_2^t\cdot\1&\cdots&B_k\\
	 \end{pmatrix}
	 ,\]
	 where $R$ is the root of $T$, $B_i$ is the quotient matrix of the distance matrix, induced by those vertices in $G$ belonging to some vertex of $T_{W_i}$, where $W_i$ is the $i$-th child of $R$ under the considered ordering of $\mathcal W$, and the vector $\vec{x}_i$ has $|W|$ in each entry corresponding to $W \in V (T_{W_i})$ for each $1\le i\le k$. Now  we apply on $D/\mathcal W$ the following elementary operations, first $r_W-2r_R\to r_W$ and then $ -r_W\to r_W$, for each $W\in\mathcal W\setminus\{R\}$, obtaining the following matrix	 

	 \[M_1=
	 \begin{pmatrix}
	 |R|-1&\vec{x}_1^t&\vec{x}_2^t&\cdots&\vec{x}_k^t\\
	 (|R|-2)\cdot\1&A_1&\0&\cdots&\0\\
	 (|R|-2)\cdot\1&\0&A_2&\cdots&\0\\
	 \vdots&\vdots&\vdots&\ddots&\vdots\\
	 (|R|-2)\cdot\1&\0&\0&\cdots&A_k\\
	 \end{pmatrix}
	 ,\]
	 
	 where the $A_i$'s are the arrow matrices of the subtrees $T_{W_i}$'s of $T$. We can transform $M_1$ into
	 
	 \[M_2=
	 \begin{pmatrix}
	 |R|-1&-\vec{a}_1^t&-\vec{a}_2^t&\cdots&-\vec{a}_k^t\\
	 \vec{b}_1&C_1&\0&\cdots&\0\\
	 \vec{b}_2&\0&C_2&\cdots&\0\\
	 \vdots&\vdots&\vdots&\ddots&\vdots\\
	 \vec{b}_k&\0&\0&\cdots&C_k\\
	 \end{pmatrix}
	 ,\]
	 such that, for each $i$, the first entry of $\vec{b}_i$ is $k_i(|R|-2)$ with $k_i\le \frac {-1} 2$,
	 \[C_i=
	 \begin{pmatrix}
	 1+|W_i|&-\vec{d}_1^t&-\vec{d}_2^t&\cdots&-\vec{d}_\ell^t\\
	 \vec{c}_1^i&C_1^i&\0&\cdots& \0\\
	 \vec{c}_2^i&\0 &C_2^i&\cdots&\0\\
	 \vdots&\vdots&\vdots&\ddots&\vdots\\
	 \vec{c}_{\ell}^i&\0&\0&\cdots&C_{\ell}^i\\
	 \end{pmatrix}
	 ,\]
	 and $\vec{a}_i$ stands for the vector having as many rows as $B_i$, a $1$ in the first column and $0$'s in the rest of its entries.
	The vector $\vec{d}_i$ has as many rows as $C_j^i$, a $1$ in the first columns and $0$'s in the rest of its entries. The block $C_j^i$ is a lower matrix, $(C_j^i)_{11}=(1+m_j)|S_j|$ with $m_j\le\frac{-1} 2$ and $S_j$ is the child of $W_i$ corresponding to the first row of $C_j^i$, and the  first entry of  $\vec{c}_j^i$ is $m_j|S_j|$. 
	
	 We will prove that there exists a sequence of elementary row operations leading $M_1$ to $M_2$. First, we do $r_R-\sum_ W r_W\to r_R$, the sum is taken among all vertices $W\in V(T)$ such that $W$ is a child of $R$. We repeat this procedure on each $T_W$ such that $h(T_W)\ge 2$  and $W$ is a child of $R$. Then we proceed with every child of the $T_W$'s and so on as long as possible. Let us call this new matrix $M_2^\prime$. Notice that entries of $M_2^\prime$ have been modified according to $M_2$ as follow $(M_2^\prime)_{WV}=|W|+1-\alpha_W |W|$ for each $W=V$ or $W$ ancestor of $V$, where $\alpha_W$ is the number of children of $W$ on $T_W$; and $(M_2^\prime)_{RR}=|R|-1-\alpha_R (|R|-2)$, where $\alpha_R$ is the number of children of $R$ on $T$.
	 We proceed by applying induction on $h(T_W)$, the height of $T_W$.  
	 
	 Base case: $h(T_W)=1$. We do $ r_W -\sum_{V}\frac{|V|}{|V|+1}\cdot r_V \to r_W$, the sum is taken over all children $V$ of $W$. Under this row operation we obtain a matrix $N$ such that $N_{WV}=0$ for every descendant $V$ of $W$, $N_{WW}=|W|+1-\alpha_W |W|$, $N_{WV}=|V|-\alpha_W |V|$  for each $V$ ancestor of $W$ distinct of $R$, and $N_{WR}=|R|-2-\alpha_W (|R|-2)$. Thus $N_{VV}=1+m_W |V|$, $N_{WV}=m_W|V|$ for each ancestor $W$ distinct of $R$ and $N_{WR}=m_W(|R|-2)$, where $m_W=1-s_W+\sum_V\frac 1 {|V|+1}$ and $s_W$ is the number of children of $W$. Hence, since $|V|\ge 6>3$,  $m_W< 1-s_W+\frac{s_W} 4$. Therefore, $s_W\ge 2$ implies $m_W<\frac{-1}{2}$. 
	 
	 Assume now, by inductive hypothesis, that we can obtain a matrix $N$ from $M'_2$, by means of elementary rows operations such that if $1\le h(T_W)<k<h(T_R)$ with $W\neq R$, $N_{WV}=0$ for each descendant $V$ of $W$, $N_{WW}=1+m_W|W|$ and $N_{WV}=m_V |V|$ with $m_V\le\frac{-1} 2$ for each ancestor $V$ of $W$ distinct of $R$, and $N_{WR}=m_R(|R|-2)$ with $m_R\le\frac{-1}{2}$. These are the only entries modified concerning to $M'_2$. Let $W'$ be a vertex of $T$ such that $1<h(T_{W'})=k$. We modify row $W'$ according to  $r_W'+\sum_{V}\frac{1}{m_V|V|+1}r_V\to r_{W'}$, where the sum is taken over all children $V$ of $W'$ such that $h(T_V)\ge 1$; and then we do $r_{W'}+\sum_{V'}\frac{1}{|V'|+1}\cdot r_{V'}\to r_{W'}$, the sum is taken over all children $V'$ of $W'$ such that $h(T_{V'})=0$. Hence the new matrix $N'$ satisfies
	 
	 \begin{align*}
			 N'_{W'W'}&=\left(|W'|+1-s_{W'}|W'|+\sum_{h(T_V)\ge 1}\frac{m_V |W'|}{m_V|V|+1}+\sum_{h(T_{V'})=0}\frac{|W'|}{|V'|+1}\right)\\
			 &\le 1+|W'|\left(1-s_{W'}+\sum_{h(T_V)\ge 1}\frac{1}{|V|-2}+\sum_{h(T_{V'})=0}\frac{1}{|V'|+1}\right)\\
			 &\le 1+|W'|\left(1-\frac 3 4 s_{W'}\right).\\
	 \end{align*}
		By the inductive hypothesis, $m_V\le \frac{-1}{2}$ for each $V$ child of $W'$ such that $h(T)<k$ and thus the first inequality holds. The last one follows from $|V|\ge 6$ for each vertex $V$ of $T$. We conclude that $N_{W'W'}<0$. Using the inductive hypothesis and reasoning as in the base case, it follows that $N'_{WV}=0$ for each descendant $V$ of $W$ and $N'_{WV}=m_V|V|$ with $m_V\le\frac{-1}2$ for each ancestor $V$ of $W$ distinct of $R$, and $M_{WR}=(|R|-2)m_R$. In particular, the result holds for each child $W$ of $R$. Hence $M_2$ can be obtained from $D/\mathcal{W}$ through elementary row operations. 
	
		Finally, use the same strategy as in the inductive hypothesis to prove our result. We can prove that, if we do $r_R+\sum_{V}\frac{1}{m_V|V|+1}r_V\to r_{R}$, where the sum is taken over all children $V$ of $R$ such that $h(T_V)\ge 1$; and then we do  $r_{R}+\sum_{V'}\frac{1}{|V'|+1}\cdot r_{V'} \to r_{R}$, where the sum is taken over all children $V'$ of $R$, we obtain a lower matrix whose main diagonal has no zero entry. 
\end{proof}

\subsection{Nullity}

Trivially perfect graphs are a superclass of threshold graphs, but, unlike threshold graphs, for every $k\ge 2$ there exists a trivially pefect graph with nullity $k$.

\begin{theorem}\label{thm: nullity} 
	Let $N=3k+r$, where $r\in\{1,2,3\}$ and $k\ge 2$ is an integer. Let $n \in \mathbb{N}$ with $n\geq \textcolor{magenta}{7}k+r$.
	Then, there exists a trivially perfect graph $G$ that has a true twin partition into $N$ sets and $|V(G)|=n$ such that the distance matrix of $G$ has nullity $\ell$, where
	\begin{itemize}
	\item $\ell=k-1$ if $n= 7k+r$ or if $r=1$ and $n\ge 7k+3$,
	\item $\ell=k$ if $r=1$ and $n=7k+2$,
	\item $\ell\in \{k-1,k\}$ otherwise.
	\end{itemize} 
\end{theorem}
\begin{proof}
Let $N$ be an integer number such that $N=3k+r$ with $r= 1,2 \ or \ 3$ and $k\ge2$.

Let $G$ be a trivially perfect graph with $|V(G)|=n\ge\textcolor{magenta}{7}k+r$, having a true twin partition $\mathcal{W}=\{R_1, \cdots, R_r, W_1,W_{1,1}, W_{1,2},W_2,W_{2,1}, W_{2,2},\ldots,W_k,W_{k,1}, W_{k,2}\}$ such that
\begin{itemize}
 \item $|W_i|=3$ for $1\le i \le k$,
 \item $|W_{i,1}|=|W_{i,2}|=2$ for $1\le i \le k$,
 \item $|R_i|\ge 1$ for $1\le i \le r$,
 \item $|R_1|+\cdots+|R_r|=n-7k$.
\end{itemize}
Let $T=(\mathcal W, E)$ a rooted clique tree of $G$, where 
\begin{itemize}
 \item $R_1$ is the root,
 \item $R_i$ is descendant of $R_1$ for $ 2 \le i\le r$, if $r\ge2$,
 \item $W_i$ is descendant of $R_1$ for $ 1 \le i\le k$,
 \item $W_{i,1}$ and $W_{i,2}$ are descendants of $W_i$ for $1 \le i\le k$.
\end{itemize}

Consider the distance matrix $D$ of $G$, Lemma \ref{lem:matrizcociente} allows us to use $D/ \mathcal W$ instead of $D$ to study its nullity. Using the same transformations as in the proof of Theorem \ref{thm: nonsingular trivially perfect graphs}, we obtain the matrix
\[M=
	 \begin{pmatrix}
	 |R_1|-1&\vec{x}_1^t&\vec{x}_2^t&\cdots&\vec{x}_k^t\\
	 (|R_1|-2)\cdot\1&A_1&\0&\cdots&\0\\
	 (|R_1|-2)\cdot\1&\0&A_2&\cdots&\0\\
	 \vdots&\vdots&\vdots&\ddots&\vdots\\
	 (|R_1|-2)\cdot\1&\0&\0&\cdots&A_k\\
	 \end{pmatrix}
	 ,\]
if $r=1$, or  
\[M=
	 \begin{pmatrix}
	 |R_1|-1          & |R_2| & \cdots & |R_r|& \vec{x}_1^t&\vec{x}_2^t&\cdots&\vec{x}_k^t\\
	 |R_1|-2 & |R_2|+1 & \0^t & 0 & \0^t& \0^t &\cdots& \0^t\\
	 \vdots  & \0 & \ddots & \0 & \0^t& \0^t &\cdots& \0^t\\
	 |R_1|-2 & 0  &  \0^t  &|R_r|+1& \0^t& \0^t &\cdots& \0^t\\
    (|R_1|-2)\cdot\1 & \0 &\cdots& \0 & A_1&\0&\cdots&\0\\
	 (|R_1|-2)\cdot\1 & \0 &\cdots& \0 & \0 &A_2&\cdots&\0\\
	 \vdots           &\vdots & \vdots & \vdots &\vdots&\vdots&\ddots&\vdots\\
	 (|R_1|-2)\cdot\1 & \0 &\cdots& \0 & \0&\0&\cdots&A_k\\
	 \end{pmatrix}
	 ,\]
if $r\ge2$, where 
$$
A_i=
\begin{pmatrix}
|W_i|+1 & |W_{i,1}| & |W_{i,2}|\\ 
|W_i| & |W_{i,1}|+1 & 0\\
|W_i| & 0 & |W_{i,2}|+1
\end{pmatrix}
=
\begin{pmatrix}
4 & 2 & 2\\ 
3 & 3 & 0\\
3 & 0 & 3
\end{pmatrix},
$$
and 
$$\vec{x}_i^t = (|W_i|, |W_{i,1}|, |W_{i,2}|)=(3,2,2) ,$$
for $1\le i\le k$. 

By elementary row operations, we obtain
\[\hat{M}=
	 \begin{pmatrix}
	 \hat{R} &\vec{\hat{x}}^t&\vec{\hat{x}}^t&\cdots&\vec{\hat{x}}^t\\
	 (|R_1|-2)\cdot\vec{\hat{y}}&\hat{A}&\0&\cdots&\0\\
	 (|R_1|-2)\cdot\vec{\hat{y}}&\0&\hat{A}&\cdots&\0\\
	 \vdots&\vdots&\vdots&\ddots&\vdots\\
	 (|R_1|-2)\cdot\vec{\hat{y}}&\0&\0&\cdots&\hat{A}\\
	 \end{pmatrix}
	 ,\]
if $r=1$, or  
\[\hat{M}=
	 \begin{pmatrix}
	   \hat{R}  & 0 & \cdots & 0& \vec{\hat{x}}^t&\vec{\hat{x}}^t&\cdots&\vec{\hat{x}}^t\\
	 |R_1|-2 & |R_2|+1 & \0^t & 0 & \0^t& \0^t &\cdots& \0^t\\
	 \vdots  & \0 & \ddots & \0 & \0^t& \0^t &\cdots& \0^t\\
	 |R_1|-2 & 0  &  \0^t  &|R_r|+1& \0^t& \0^t &\cdots& \0^t\\
    (|R_1|-2)\cdot\vec{\hat{y}} & \0 &\cdots& \0 & \hat{A}&\0&\cdots&\0\\
	 (|R_1|-2)\cdot\vec{\hat{y}} & \0 &\cdots& \0 & \0 &\hat{A}&\cdots&\0\\
	 \vdots           &\vdots & \vdots & \vdots &\vdots&\vdots&\ddots&\vdots\\
	 (|R_1|-2)\cdot\vec{\hat{y}} & \0 &\cdots& \0 & \0&\0&\cdots&\hat{A}\\
	 \end{pmatrix}
	 ,\]
if $r\ge2$, where 
$$
\hat{R}=
\begin{cases}
 (|R_1|-1) - \frac{4k}{3}\cdot(|R_1|-2) & \text{if } r=1\\
 (|R_1|-1) - (\frac{4k}{3} + \frac{|R_2|}{|R_2|+1}+\cdots+\frac{|R_r|}{|R_r|+1})\cdot(|R_1|-2) & \text{if } r\ge2,
\end{cases}
$$
$$
\hat{A}=
\begin{pmatrix}
0 & 0 & 0\\ 
3 & 3 & 0\\
3 & 0 & 3
\end{pmatrix}, \ \
\vec{\hat{x}}^t = (-1,0,0), \text{ and} \ \
\vec{\hat{y}}^t = \left(-\frac{1}{3},1,1\right).$$
 
Notice that the $(r+1+3i)$-th row corresponds to the first row of $\vec{\hat{y}}$ and $\hat{A}$, and thus it equals
\[
\begin{pmatrix}
-\frac{|R_1|-2}{3}&0&0&\cdots &0
\end{pmatrix}.
\]
Thus, the nullity of $\hat{M}$ is at least $k$ if $|R_1|=2$, and at least $k-1$ otherwise.
Furthermore, it is easy to check that the rest of the rows form a linearly independent set,
and that this set does not generate the $(r+1+3i)$-th row if $|R_1|\neq 2$.
The result now follows from the fact that $1\leq |R_1|\leq n-7k-r+1$, and that
if $r=1$, then $|R_1|=n-7k-r+1$.
\end{proof}

In the proof of Theorem \ref{thm: nullity}, we assign values to $(|W_i|, |W_{i,1}|, |W_{i,2}|)$ so that $A_i$ has nullity $1$, for all $1\le i \le k$. We can see that the matrix $A_i$ has nullity $1$ if and only if $(|W_i|, |W_{i,1}|, |W_{i,2}|)$ is one of $(2,3,3)$, $(2,2,5)$, $(2,5,2)$, $(3,2,2)$, $(3,1,5)$, $(3,5,1)$, $(4,1,3)$, $(4,3,1)$, $(6,1,2)$, and $(6,2,1)$, for $1\le i \le k$. In particular we use $(3,2,2)$ because with this choice we obtain the minimum lower bound for the number of vertices.

\section{Conclusion and further research}\label{sec: conclusions}

The proof of Theorem~\ref{thm: finite number of graphs with rank k} presents an upper bound for the number of graphs with distance rank equals $k$ in terms of the Ramsey number $R(k)$. Nevertheless, this upper bound seems to be far from being tight. Indeed, $\lfloor f(3,R(3))\rfloor= 186$, and the number of connected graphs with distance rank $3$ is equal to three. It would be interesting to find a tighter upper bound for the number of connected graphs with distance rank $k$. In Theorem~\ref{thm: nulity of threshold graphs}, we prove that a connected threshold graph has nullity at most one. We also present a family of infinite power sequences giving place to an infinite family of connected threshold graphs with nullity one. A challenging problem is characterizing those connected threshold graphs with nullity equal to zero or one. Unlike threshold graphs, for each integer $k\ge 2$ there exists a trivially perfect graph with nullity equal to $k$, see Theorem~\ref{thm: nullity}. Notice that, Theorem~\ref{thm: nonsingular trivially perfect graphs} guarantees that if each set of the twin partition of a trivially perfect graph is big enough, then its distance matrix is nonsingular. Consequently, connected threshold graphs with nullity one have a small set in their twin partition, as they are a subclass of trivially perfect graphs. 

\section*{Acknowledgments}

Ezequiel Dratman and Luciano N. Grippo acknowledge partial support from ANPCyT PICT 2017-1315. The first two authors and Ver\'onica Moyano were partially supported from Universidad Nacional de General Sarmiento, grant UNGS-30/1135.
Adri\'{a}n Pastine ackowledges partial suppport from Universidad Nacional de San Luis, Argentina, grants PROICO 03-0918
and PROIPRO 03-1720, and from ANPCyT grants PICT-2020-SERIEA-04064 and PICT-2020-SERIEA-00549.

This article was conceived during a visit of the fourth author to Universidad Nacional de General Sarmiento and he would like to thank them for their hospitality.


\begin{thebibliography}{10}

\bibitem{BapatKirklandandNeumann2005}
R.~Bapat, S.~J. Kirkland, and M.~Neumann.
\newblock On distance matrices and {L}aplacians.
\newblock {\em Linear Algebra Appl.}, 401:193--209, 2005.

\bibitem{ChengandLiu2007}
C.~Bo and B.~Liu.
\newblock On the nullity of graphs.
\newblock {\em Electron. J. Linear Algebra}, 16:60--67, 2007.

\bibitem{CHY2011}
G.~J. Chang, L.-H. Huang, and H.-G. Yeh.
\newblock A characterization of graphs with rank 4.
\newblock {\em Linear Algebra Appl.}, 434(8):1793--1798, 2011.

\bibitem{CHY2012}
G.~J. Chang, L.-H. Huang, and H.-G. Yeh.
\newblock A characterization of graphs with rank 5.
\newblock {\em Linear Algebra Appl.}, 436(11):4241--4250, 2012.

\bibitem{Corneil81}
D.~G. Corneil, H.~Lerchs, and L.~S. Burlingham.
\newblock Complement reducible graphs.
\newblock {\em Discrete Appl. Math.}, 3(3):163--174, 1981.

\bibitem{Dratmanetal2021}
E.~Dratman, L.~N. Grippo, M.~D. Safe, C.~M. da~Silva, Jr., and R.~R.
  Del-Vecchio.
\newblock The determinant of the distance matrix of graphs with blocks at most
  bicyclic.
\newblock {\em Linear Algebra Appl.}, 614:437--454, 2021.

\bibitem{Gol78}
M.~C. Golumbic.
\newblock Trivially perfect graphs.
\newblock {\em Discrete Math.}, 24(1):105--107, 1978.

\bibitem{Golumbic2004}
M.~C. Golumbic.
\newblock {\em Algorithmic graph theory and perfect graphs}, volume~57 of {\em
  Annals of Discrete Mathematics}.
\newblock Elsevier Science B.V., Amsterdam, second edition, 2004.
\newblock With a foreword by Claude Berge.

\bibitem{GrahamLovasz1978}
R.~L. Graham and L.~Lov\'{a}sz.
\newblock Distance matrix polynomials of trees.
\newblock In {\em Probl\`emes combinatoires et th\'{e}orie des graphes
  ({C}olloq. {I}nternat. {CNRS}, {U}niv. {O}rsay, {O}rsay, 1976)}, volume 260
  of {\em Colloq. Internat. CNRS}, pages 189--190. CNRS, Paris, 1978.

\bibitem{GraphamandPollak1973}
R.~L. Graham and H.~O. Pollak.
\newblock On the addressing problem for loop switching.
\newblock {\em Bell System Tech. J.}, 50:2495--2519, 1971.

\bibitem{Merris1994}
R.~Merris.
\newblock Laplacian matrices of graphs: a survey.
\newblock volume 197/198, pages 143--176. 1994.
\newblock Second Conference of the International Linear Algebra Society (ILAS)
  (Lisbon, 1992).

\bibitem{Ramsey1929}
F.~P. Ramsey.
\newblock On a {P}roblem of {F}ormal {L}ogic.
\newblock {\em Proc. London Math. Soc. (2)}, 30(4):264--286, 1929.

\bibitem{Spencer1975}
J.~Spencer.
\newblock Ramsey's theorem---a new lower bound.
\newblock {\em J. Combinatorial Theory Ser. A}, 18:108--115, 1975.

\bibitem{west2001}
D.~B. West.
\newblock {\em Introduction to graph theory}.
\newblock Prentice Hall, Inc., Upper Saddle River, NJ, 2001.

\bibitem{Chang96}
J.-H. Yan, J.-J. Chen, and G.~J. Chang.
\newblock Quasi-threshold graphs.
\newblock {\em Discrete Appl. Math.}, 69(3):247--255, 1996.

\end{thebibliography}
\end{document}